\documentclass[a4paper, 11pt, parskip=half]{amsart}

\usepackage{custom-template}
\usepackage{mathtools}
\usepackage{tabularx}
\usepackage{array}
\usepackage{shortcuts}
\usepackage{stmaryrd}

\begin{document}
	\title{On Lower Bounds for sums of Fourier Coefficients of Twist-Inequivalent Newforms
}
	\date{}
	\author{Moni Kumari, Prabhat Kumar Mishra, and Jyotirmoy Sengupta}

	\address{Department of Mathematics, Indian Institute of Technology Jammu, Jagti, PO Nagrota, NH-44 Jammu 181221, J \& K, India\vspace*{-5pt}}
	\email{moni.kumari@iitjammu.ac.in\vspace*{-6pt}}
	\email{2022rma0025@iitjammu.ac.in\vspace*{-6pt}}

    \address{School of Mathematical \& Computational Sciences, Indian Association for the Cultivation of Science, 2A \& B Raja S C Mullick Road, Kolkata 700032, West Bengal, India \vspace*{-5pt}}
    \email{jyotirmoysengupta671@gmail.com}

	
	\begin{abstract}
In this article, we address the lower bounds for the sums
$a_f(p)+a_g(p)$ of the $p$th Fourier coefficients of two twist-inequivalent, non-CM normalized newforms 
$f$ and $g$. 
Our main result shows that for such forms with integer Fourier coefficients, the largest prime factor of $a_f(p)+a_g(p)$ satisfies $ P(a_f(p)+a_g(p))> (\log p)^{1/14}(\log \log p )^{3/7-\epsilon}$
for almost all primes 
$p$ and for any $\epsilon>0$. Beyond primes, we apply Brun’s sieve to show that a similar phenomenon holds for a set of positive integers with natural density 1.
The main result is further strengthened under the Generalized Riemann Hypothesis, where we establish exponential growth for $|a_f(p)+a_g(p)|$ in terms of $p$. 

Additionally, we derive an interesting result related to the multiplicity one theorem, demonstrating that if the sum $a_f(p)+a_g(p)$
is small for a positive-density subset of primes, then $f$ and $g$
must be twist-equivalent by a quadratic character. 
\end{abstract}
	\maketitle

\section{Introduction and results}
Let ${S}_{k}(N)$ denote the space of cusp forms of trivial nebentypus, even weight $k\ge 2$, and level $N$. The study of Fourier coefficients of cusp forms has long been a central theme in number theory, yet many aspects remain mysterious.

A landmark result in this area is the Ramanujan-Petersson conjecture, which provides an upper bound on the magnitude of these coefficients. This conjecture was proven by Deligne in 1974. Specifically, if
$$f(z)=\sum_{n=1}^{\infty}a_f(n)q^n \in S_k(N), \quad q=e^{2\pi i z}, \quad {\rm Im}(z)>0$$
is a holomorphic, cuspidal, normalized Hecke newform without complex multiplication (CM), then Deligne’s theorem asserts that for any prime $p$, 
\begin{equation}\label{deligne bound}
    |a_f(p)|\le 2 p^{\frac{k-1}{2}}.
\end{equation}
Since then, extensive research has examined the sharpness of this bound and its implications.

In addition to upper bounds, the lower bounds for the magnitude of these Fourier coefficients, when they are nonzero, also present an intriguing area of investigation. 
These bounds play a crucial role in understanding the structure of these coefficients and their applications in various areas of number theory and harmonic analysis.

A key conjecture in this direction, known as the Atkin-Serre conjecture \cite{atkser}, states 
 that for a  non-CM normalized newform $f\in {S}_{k}(N)$ of weight $k\ge 4$  and for any 
 $\epsilon>0$ there exists a positive constant $c_{\epsilon,f}$  such that for all sufficiently large primes 
$p$
\begin{equation}\label{atkser}
     |a_f(p)|\ge c_{\epsilon,f} ~ p^{\frac{k-3}{2}-\epsilon}.   
\end{equation}
Gafni, Thorner, and Wong \cite[Thm. 1.1]{gaf} confirmed this conjecture for a density-one subset of primes, proving the stronger bound
 \[
            |a_f(p)| > 2 p^{\frac{k-1}{2}} \frac{\log\log p} { \sqrt{\log p}},
\]
for almost all primes.

When Fourier coefficients are integers, their size can alternatively be measured by their largest prime factor, where for a non-zero integer $n$, $P(n)$ denotes the largest prime factor of $|n|$. 
 In this direction, Murty and Murty \cite[Thm. 6.2]{rmv}, assuming the generalized Riemann hypothesis (GRH), established that for any $\epsilon>0$, $ P(a_f(p)) > e^{( \log p)^{1-\epsilon}}$ for almost all primes.
 Unconditionally, Murty, Murty, and Saradha \cite[Thm. 6.2]{MKS} showed that $P(a_f(p)) > e^{( \log \log p)^{1-\epsilon}}$ holds for almost all primes $p$.
 More recently, Bilu, Gun, and Naik \cite[Thm. 1]{gun} further strengthened this bound to
 $$P(a_f(p))>(\log p)^{1/8} (\log \log p)^{3/8-\epsilon},$$
 which holds for almost all primes.
Inspired by these results, this article investigates the sum $a_f(p)+a_g(p)$ for two distinct non-CM newforms $f\in S_k(N_1)$ and $S_k(N_2)$. Our primary goal is to establish lower bounds on the largest prime factor of this sum. 

To formulate this precisely, we introduce the notion of twist-inequivalent newforms. We say  two non-CM newforms  $f$ and $g$ are twist-inequivalent if there does not exist any primitive Dirichlet character $\chi$ such that $f= \chi \otimes g$. 
In other words, $f$ and $g$ are not related by a Dirichlet character twist.
This concept plays a crucial role in our study, leading to new results on the prime factorization of sums of Fourier coefficients of newforms.
\begin{theorem}\label{main}
     Let $f\in S_k(N_1)$ and  $g\in S_k(N_2)$ be two non-CM normalized newforms with trivial nebentypus that are twist-inequivalent and having integer Fourier coefficients $a_f(n)$ and $a_g(n)$, respectively. Then for any $\epsilon >0$, 
     \[
            P(a_f(p)+a_g(p))> (\log p)^{1/14}(\log \log p )^{3/7-\epsilon},
     \]
   holds for almost all primes.
\end{theorem}
Utilizing Theorem \ref{main} and Brun’s sieve, we extend this result to the convolution function  $a_f\ast a_g$
, showing that the phenomenon persists for almost all positive integers in terms of natural density. Specifically, the set of integers $n$
 where either  $(a_f\ast a_g)(n)=0$
 or its largest prime factor satisfies a similar lower bound has density 1, as stated in the following theorem.
\begin{theorem}\label{over n}
   Let $f$ and $g$ as in Theorem \ref{main}. Then for any $\epsilon>0$, the set
   $$\{n:(a_f\ast a_g)(n)=0~ {\rm{or}}~P((a_f\ast a_g)(n))>(\log n)^{1/14}(\log \log n )^{3/7-\epsilon} \}$$
   has natural density 1.
\end{theorem}

\begin{remark}
Using a result of Erd\Horig{o}s \cite{erdos} on $\mathcal{B}$-free numbers, combined with Lemma \ref{same_coefficients}, one can establish that the set $\{n: (a_f\ast a_g)(n) \neq 0\}$
has positive density. Consequently, this implies that the set 
$$\{n: P((a_f\ast a_g)(n))>(\log n)^{1/14}(\log \log n )^{3/7-\epsilon}\}$$
also has positive density.
\end{remark}

Another interesting consequence of Theorem \ref{main} is the following result, which gives a variation of the multiplicity one theorem demonstrating that if the sum $a_f(p)+a_g(p)$
is small for a positive-density subset of primes, then $f$ and $g$
must be twist-equivalent by a quadratic character. 
\begin{corollary}\label{character}
Let $f$ and $g$ be two non-CM normalized newforms of trivial nebentypus with integer Fourier coefficients. If the set
$$\{p:|a_f(p)+a_g(p)|\leq (\log p)^{1/14}(\log \log p )^{3/7-\epsilon}\}$$
has a positive upper density for some $\epsilon>0$, then 
$f$ and $g$ are twist-equivalent by a quadratic character.
\end{corollary}
This result should be compared with Ramakrishnan’s theorem in the appendix of \cite{duke}. Ramakrishnan assumes that if $a_f^2(p) = a_g^2(p)$ for a subset of primes with density $\epsilon>0$, no matter how small $\epsilon$ is, then the relation $$a_f(p) = \chi(p)a_g(p)$$
holds for all $p\nmid N_1N_2$, where $\chi$ is a quadratic character with conductor dividing $N_1N_2$.
Our result in Corollary \ref{character} states that if  $f$ and $g$ are two non-CM normalized newforms satisfying $\{p:|a_f(p)+a_g(p)|\leq (\log p)^{1/14}(\log \log p )^{3/7-\epsilon}\}$ with a positive upper density for some $\epsilon >0$, then the same conclusion holds: that is, $a_f(p) = \chi(p)a_g(p)$ for all $p\nmid N_1N_2$. Thus, while our starting assumption differs from Ramakrishnan’s, we arrive at the same conclusion. However, our approach requires the additional assumption that 
$f$ and $g$ have integer Fourier coefficients.

Finally, assuming GRH, we obtain the following refined lower bounds on the growth of  $|a_f(p)+a_g(p)|$
 and its largest prime factor. 
\begin{theorem}\label{under GRH}
 Let $f\in S_k(N_1)$ and  $g\in S_k(N_2)$ be two non-CM normalized newforms with trivial nebentypus that are twist-inequivalent and having integer Fourier coefficients $a_f(n)$ and $a_g(n)$, respectively. Then under GRH, for almost all primes $p$ and any $\epsilon>0$, we have
       \begin{itemize}
           \item[(a)] 
           $\log P(a_f(p)+a_g(p))\geq \frac{\log p}{e^{3(\log \log p)^{1/2+\epsilon}}}\ge (\log p)^{1-\epsilon}.$
           
\item[(b)]
$\log |a_f(p)+a_g(p)| \ge (\log \log p)^{1/2+\epsilon} \frac{ \log p}{e^{3(\log \log p)^{1/2+\epsilon}}}.$
\end{itemize}  
\end{theorem}
At the end, we make some important comments about the paper.
\begin{enumerate}
\item 
Following the line of proof of \cite[Theorem 3]{gun}, one can derive an analogous result for the sum $a_f(p)+a_g(p)$, which yields a substantial improvement over part $(a)$ of \Cref{under GRH}. However, we retain the formulation of \Cref{under GRH}, as we aim to study the normal order of the function $a_f(p)+a_g(p)$ (see Sections  \ref{section: normal order},  \ref{sec: proof uner grh} for further details).

\item 
We remark that in Theorem \ref{main} and all subsequent results about primes in this paper remain valid if we replace $a_f(p)+a_g(p)$ with $a_f(p)-a_g(p).$
\end{enumerate}

\subsection*{Notation and conventions} 
Throughout the article, the newforms considered are of trivial nebentypus. 
By GRH, we mean the generalized Riemann hypothesis for all Artin $L$-functions. The letters $\ell$, $p$, and $q$ always denote positive rational primes unless stated otherwise. 
For arithmetic functions $\phi_1, \phi _2:\N\rightarrow \C$ we define the convolution 
$(\phi_1 \ast \phi_2)(n):=\sum_{km=n}\phi_1(k)\phi_2(m).$
The notations $O$ (big-O) and $o$ (small-o) are the usual Landau notations.
For $x\geq2, \pi(x)$ denotes the number of primes up to $x$.
For a set $A$,  $A(x)$ denotes the number of elements in $A$ up to  $x$. Saying that a property holds for almost all primes, we mean that the set of primes satisfying the given property has the natural density 1.

\section{Preliminaries} 
This section includes essential prerequisites such as the joint Sato-Tate theorem and product Galois representations linked to two non-CM cuspidal normalized newforms.

\subsection{Joint Sato-Tate for Fourier coefficients} 
Let  $f\in {S}_k(N)$ be a non-CM holomorphic cuspidal normalized newform with Fourier coefficients $a_f(n)$, then from \ref{deligne bound}, 
$$|a_f(p)|\le 2p^{(k-1)/2}.$$
The Sato-Tate conjecture, now proved (see \cite[Thm. B]{lamb}), states that the sequence $\{\frac{a_f(p)}{p^{k-1/2}}: p ~{{\rm prime}}\}$ is equidistributed in the interval $[-2,2]$ with respect to the Sato-Tate measure, $d\mu_{ST}(t) = \sqrt{1-\frac{t^2}{4}}\,dt.$
More precisely, for any interval $I\subset [-2,2]$, 
\[
       \lim_{x\rightarrow \infty} \frac{\left\vert \left\{p\le x :  p\nmid N, ~\frac{a_f(p)}{p^{(k-1)/2}}\in I\right\}\right\vert} {\pi(x)} = \mu_{ST} (I). 
\] 
 This was further extended to a pair of cusp forms which we call the joint Sato-Tate theorem 
 (see \cite{wong}). More precisely, if $f\in {S}_k(N_1)$ and $g\in {S}_k(N_2)$ are two non-CM normalized cuspidal newforms which are twist-inequivalent and $B\subset [-2,2]\times [-2,2]$ is a Borel set, then 
\begin{equation}\label{joint sato tate}
        \lim_{x\rightarrow \infty} \frac{\left\vert \left\{p\le x : p\nmid N_1N_2, \,\Big(\frac{a_f(p)}{p^{(k-1)/2}}, \frac{a_g(p)}{p^{(k-1)/2}}\Big)\in B\right\}\right\vert}{\pi(x)} = \nu_{ST} (B),
\end{equation}
where $\nu_{ST}$ is the measure on $[-2,2]\times [-2,2]$ given by $d\nu_{ST}(s,t) = \sqrt{(1-\frac{s^2}{4})(1-\frac{t^2}{4})}\,ds dt$.

\subsection{Chebotarev density theorem}\label{S_basic}
Let $K$
 be a finite Galois extension of $\Q$ with Galois group $G$. For a prime $p$ unramified in $K$, let ${\rm Frob}_p$ denote a Frobenius element at $p$ in $G$.
 Given a conjugation-invariant subset $C \subset G$, define $$
\pi_C(x):=\big|\{p\le x: p~ {\rm is ~ unramified ~in}~ K ~{\rm and ~Frob}_p\in C\}\big|.
$$
The Chebotarev Density Theorem asserts that
\begin{equation}\label{cdt main statement}
\lim_{x \rightarrow \infty}\frac{\pi_C(x)}{{\rm Li}(x)}= \frac{|C|}{|G|}, ~~~{\rm where ~ Li}(x)=\int_{2}^{x}\frac{dt}{\log t}\sim \frac{x}{\log x}~~{\rm as}~~ x \rightarrow\infty.\end{equation}
Denoting the degree and absolute discriminant of $K$ by $n_K$ and $d_K$, respectively, Lagarias and Odlyzko established effective versions of this theorem. 
Unconditionally, the version given in \cite[Eq. 1.7]{lo} ensures the existence of an absolute, effective, and computable constant $c'>0$ such that  for $x \geq \exp(10 n_K (\log d_K)^2), $
	\begin{equation}\label{cdt}
		\left|\pi_C(x)- \frac{|C|}{|G|}{\rm Li}(x) \right| \leq  \frac{|C|}{|G|}{\rm Li}(x^\beta)+O\left( \Vert C\Vert x ~{\rm exp}\left(-c'\sqrt{\frac{\log x}{n_K}}\right) \right),
	\end{equation}
    where  $\Vert C\Vert$ denotes the number of conjugacy classes in $C$ and  $\beta$ is the possible Landau-Siegel zero of $\zeta_K(s)$ within the strip 
	$$
	1-\frac{1}{4\log d_K}\le \Re (s)<1.
	$$
If such a zero does not exist, the corresponding term in \eqref{cdt} is omitted. The best-known bound for $\beta$ due to Stark \cite[Eq. (27)]{sta} is	\begin{equation}\label{beta_stark}
		\beta \le 1-\frac{c_0}{d_K^{1/n_K}},
	\end{equation}
	where $c_0>0$ is an effective constant.
    Assuming GRH for the Dedekind zeta function $\zeta_K(s)$, in \cite[Thm. 1.1]{lo}, Lagarias and Odlyzko proved that for $x\geq 2$,
\begin{equation}\label{CDT under GRH}
			\pi_C(x)= \frac{|C|}{|G|}\pi(x)+O\bigg(\frac{|C|}{|G|}x^{{1}/{2}}(\log d_K+n_K\log x)\bigg).
		\end{equation}
    For applications, it is often useful to obtain upper and lower bounds for $\pi_C(x)$ of order Li$(x)$
 in a suitable range of $x$.
 Lagarias and Odlyzko \cite{LMO} made significant progress in this direction, and various mathematicians have further improved their results. In particular, Thorner and Zaman \cite[P. 2]{tz} established the existence of an absolute effective constant $A$ such that
    \begin{equation}\label{thorner}
    \pi_C(x)\ll \frac{|C|}{|G|}{\rm Li}(x)\quad  {\rm for}\,\, x\geq (d_Kn_K^{n_K})^A.
 \end{equation}

\subsection{mod-$h$ Galois representations}\label{seclh} 
	Let $G_{\Q}= {\rm Gal}(\bar\Q/\Q)$ be the absolute Galois group of an algebraic closure $\bar\Q$ of $\Q$. Let $k\ge 2, N\ge 1$ be positive integers, and $ \ell $ be a prime. Suppose $f\in S_k(N)$  is a normalized newform
	with Fourier coefficients $a_f(n)$. {The works of Eichler, Shimura, and Deligne (see \cite{del})} give the existence of a 
	two-dimensional  continuous, odd, and irreducible Galois representation
	\begin{equation*}
		\rho_{f, \ell }:G_{\Q}\rightarrow {\rm GL}_2(\Z_\ell )
	\end{equation*}
	which is unramified at $p\nmid N\ell$.  If ${\rm{Frob}}_p$ denotes a Frobenius element corresponding to a prime $p\nmid N\ell$, then the representation $\rho_{f, \ell }$ has the property that 
	$$ {\rm{tr}}(\rho_ {f,\ell}({\rm{Frob}}_p))=a_f(p), \quad {\rm {det}}(\rho_ {f,\ell}({\rm{Frob}}_p))=p^{k-1}.$$
	By reduction and semi-simplification, we obtain a mod-$\ell$ Galois representation, namely 
	$$ {\bar{\rho}}_ {f,\ell}:G_{\Q}\rightarrow {\rm GL}_2(\F_\ell), $$
	where $\F_{\ell}:=\Z/\ell\Z$.
	
	Let $h=\prod_{j=1}^{t}\ell_j^{n_j}$ be a positive integer.
	Using the $\ell_j$-adic representations attached to $f$, 
	we consider an $h$-adic representation given by  products of mod-$\ell_j$ representations
	\begin{equation*}
		{\rho}_{f,h}:  G_\Q \rightarrow {\rm GL}_2\big(\prod_{1\le j\le t}\Z_{\ell_j}\big).
	\end{equation*}
	For each $1\le j\le t$, we have the natural projection $\Z_{\ell_j}\twoheadrightarrow {\Z/{\ell_j^{n_j}}\Z}$, 
	and hence 
	we obtain a mod-$h$ Galois representation given by 
		$$ {\bar{\rho}}_{f,h} :G_\Q \rightarrow
	{\rm GL}_2\big(\prod_{1\le j\le t}{{\Z/{\ell_j^{n_j}}\Z}}\big) \xrightarrow{\cong} {\rm GL}_2(\Z/h\Z).$$
	Furthermore, if $p\nmid Nh$ is a prime, then $\bar{\rho}_{f,h}$ 
	is unramified at $p$ and
	$${ {\rm tr}\left(\bar{\rho}_{f,h}\lb{\rm Frob}_p\rb\right)\equiv a (p) \pmod h,~~~~ \quad
	{\rm det}\left(\bar{\rho}_{f,h}\lb{\rm Frob}_p\rb\right)\equiv p^{k-1}\pmod h.} $$
	Let $f \in S_{k}(N_1)$ and $g \in S_{k}(N_2)$ be  two twist-inequivalent non-CM normalized newforms having Fourier coefficients $a_f(n)$ and $a_g(n)$, respectively.
	Then one can consider the product 
	representation $\bar{\rho}_h$ of $\bar{\rho}_{f,h}$ and $\bar{\rho}_{g,h}$, defined by
	\begin{align*}
		\bar{\rho}_h: G_\Q & \rightarrow {\rm GL}_2(\Z/h\Z)\times {\rm GL}_2(\Z/h\Z),\\
		\sigma & \mapsto (\bar{\rho}_{f,h}(\sigma), \bar{\rho}_{g,h}(\sigma)) \notag.
	\end{align*}
	Let $\mathscr{A}_h$ denote the image of $G_\Q$ under $\bar{\rho}_h$.
	By the fundamental theorem of Galois theory,  the fixed field of ${\rm ker}(\bar\rho_h)$, say $L_h$, is a finite 
	Galois extension of $\Q$ and
	\begin{equation}\label{image}
		{\rm Gal}(L_h/ \Q) \cong \mathscr{A}_h.
	\end{equation}
	Let $\mathscr{C}_h$ be the subset of $\mathscr{A}_h$ defined by
	\begin{equation}\label{C_h}
		\mathscr{C}_h=\{(A,B)\in \mathscr{A}_h: {\rm tr}(A)+{\rm tr}(B)=0\}.
	\end{equation}

	We now define the following function on the set of positive integers which will play an important role throughout the paper. For an integer $h > 1$, let
	\begin{equation}\label{delta_definition}
		\delta(h):=\frac{|\mathscr{C}_h|}{|\mathscr{A}_h|}
	\end{equation}
	and $\delta(1):=1$. 
	Since the trace of the 
	image of complex conjugation is always zero, 
	$\mathscr{C}_h\neq \phi$, and hence $\delta(h)>0$ for every integer $h$.

The size of the set $\mathscr{A}_h$ has been well studied in \cite[Lemma 3.1, 3.2]{kk}. We also note that if $h=\ell$ is a sufficiently large prime, then
\begin{equation}\label{eq: size of Cl}
   |\mathscr{C}_{\ell}| =  \frac{\ell^6}{{\rm gcd}(\ell-1,k-1)} + O\left(\ell^5\right),
\end{equation}
which can be proved exactly along the same lines as the proof of  \cite[Lemma 3.3]{kk}, where the corresponding set is $\{(A,B)\in \mathscr{A}_h: {\rm tr}(A)-{\rm tr}(B)=0 \}$.

We establish the following result concerning the multiplicative behaviour of $\delta$.
	\begin{proposition}\label{delta_multiplicative_large_prime}
		For  large primes $\ell_1, \ell_2$ with $\ell_1 \neq \ell_2$, we have
		$$
		\delta(\ell_1 \ell_2)=\delta(\ell_1) \delta(\ell_2).
		$$
	\end{proposition}
    \begin{proof}
    From the definition of $\delta$, it suffices to show that $$|\mathscr{A}_{\ell_1\ell_2}| = |\mathscr{A}_{\ell_1}||\mathscr{A}_{\ell_2}| \quad {\rm and} \quad |\mathscr{C}_{\ell_1\ell_2}| = |\mathscr{C}_{\ell_1}||\mathscr{C}_{\ell_2}|.$$
       We first establish the equality for $|\mathscr{A}_{\ell_1\ell_2}|$. Note that from \cite[Theorem 3.2.2]{loeff}, it follows that
    \begin{equation*}\label{eq: Ah}
        \mathscr{A}_{\ell_1\ell_2} = \{ (A,B): {\rm GL}_2(\Z/\ell_1\ell_2\Z)\times  {\rm GL}_2(\Z/\ell_1\ell_2\Z) : {\rm det}(A) = {\rm det}(B) \in ((\Z/\ell_1\ell_2\Z)^{\times})^{k-1} \}.
    \end{equation*}
    Consider the map \begin{align*}
    \psi:  \mathscr{A}_{\ell_1\ell_2}&\rightarrow \mathscr{A}_{\ell_1}\times \mathscr{A}_{\ell_2} \\
      (A,B) &\mapsto ((A_1,B_1),(A_2,B_2)),\end{align*}
    where $(A_1,B_1) = (A\pmod{\ell_1}, B\pmod{\ell_1})$ and $(A_2,B_2) = (A\pmod{\ell_2}, B\pmod{\ell_2})$. Then, using the Chinese Remainder Theorem, it is easy to see that the map $\psi $
    is bijective, and therefore 
    $|\mathscr{A}_{\ell_1\ell_2}| = |\mathscr{A}_{\ell_1}||\mathscr{A}_{\ell_2}|$.
    
    Further, restricting $\psi$ to the set $\mathscr{C}_{\ell_1\ell_2}$, we see that the image of $\psi$ lies inside  $\mathscr{C}_{\ell_1}\times \mathscr{C}_{\ell_2}$, and again the map is bijective. Hence,  $|\mathscr{C}_{\ell_1\ell_2}| = |\mathscr{C}_{\ell_1}||\mathscr{C}_{\ell_2}|$.
\end{proof}
   The final result in this section concerns the growth of the density function $\delta$.
 We summarize this in the following proposition.  
\begin{proposition}\label{asymptotic_delta}
		For any prime $\ell$ and $n \in\N$, we have
		\begin{equation*}
        \delta(\ell)=\frac{1}{\ell}+O\bigg(\frac{1}{\ell^2}\bigg)\quad {\rm and } \quad 
			\delta(\ell^n)=O\bigg( \frac{1}{\ell^n}\bigg).
		\end{equation*}
	\end{proposition}
\begin{proof}
For sufficiently large $\ell$, the first assertion follows from \cite[Prop. 3.6]{kk}.
For exceptional primes $\ell$, we still have
$$\delta(\ell) = \frac{1}{\ell} + O\bigg(\frac{1}{\ell^2}\bigg)$$
since $\mathscr{C}_\ell$ is a subset of $\mathscr{A}_\ell$  with an additional constraint, the sum of the trace is zero. This restriction reduces $|\mathscr{C}_\ell|$ by a factor corresponding to a degree 
$1$ decrease in the exponent of $\ell$.

To prove the second assertion, it suffices to establish the result for all sufficiently large primes 
$\ell$. According to \cite[Lemma 3.2]{kk}, for all large  $\ell$ 
and for each $n\ge 1$, we have
                \begin{equation*}
                    |\mathscr{A}_{\ell^n}| = c \ell^{7(n-1)} (\ell-1)^3 (\ell^2+\ell)^2,
                \end{equation*}
                    where $c$ is a constant. Using the bound 
                   $ |\mathscr{C}_\ell|\ll {\ell^6}$ from  \cite[Eq. 3-17]{kk}, we immediately obtain
$$\delta(\ell^n) \ll \frac{1}{\ell^n}.$$
\end{proof}

\section{Results for $\pi_{f,g}(x,h)$ and $\pi_{f,g}^*(x,h)$}
	Let $f$ and $g$  be two non-CM normalized newforms as in the previous section.
	For a positive integer $h$
	and a real number $x\ge 2$, define
	\begin{equation}\label{pi_definition}
		\pi_{f,g}(x,h):=\sum_{\substack{p\le x,\, (p, h N)=1 \\ h|(a_f(p)+a_g(p))}}1.
	\end{equation}
	To obtain an asymptotic formula for $\pi_{f,g}(x,h)$, our aim is to apply the Chebotarev density theorem
	for the finite Galois extension $L_h$ of $\Q$, $L_h$ being the fixed field of the kernel of the mod-$h$ representation
	$\bar\rho_h=(\bar{\rho}_{f,h}, \bar{\rho}_{g,h})$.
	Clearly, the representation $\bar\rho_h$ is unramified at a prime $p$ such that  $(p, hN)=1$, where $N={lcm}(N_1,N_2)$, and hence  $p$ is unramified in $L_h$.  Moreover, for such a prime $p$
	$$ 
	{\rm tr}\left(\bar{\rho}_{h}({\rm Frob}_p)\right)\equiv (a_f (p) \pmod h,~ a_g(p)\pmod h).
	$$
	Thus, we can write
	\begin{equation}\label{pi_definition1}
		\pi_{f,g}(x,h)
		= |\{p\le x: p {\rm ~unramified~in~} L_h, \bar{\rho}_h\left({\rm Frob}_p \right) \in \mathscr{C}_h \}| + O(1),
	\end{equation}
	where $\mathscr{C}_h$ is defined by  \eqref{C_h} and the term $O(1)$ is to account for the presence of possible prime divisors of $N$ at which $\bar\rho_h$ is unramified. Note that  $\bar\rho_h$ is ramified at primes $p|h$ because a non-trivial power of the mod $p$ cyclotomic character is a component of its determinant, which is ramified at $p$. Next, we state the main result of this section, which estimates $\pi_{f,g}(x,h)$. 
    
	\begin{proposition}\label{asymptotic_pi}
		Let $f\in S_{k}(N_1)$ and $g\in S_{k}(N_2)$ be as before.  Let $N={lcm}(N_1,N_2)$ and  $h\ge1$ be an integer. 	Then we have
		\begin{enumerate}\label{cdt v1}
			\item[(a)]
			For  $\log x \gg h^{21}(\log (hN))^2$ 
			\begin{equation}\label{piprop}
				\pi_{f,g}(x,h)  = \delta(h)  {\rm Li}(x) +O\left( \delta(h) {\rm Li}(x^\beta)\right)+O\left( h^6x ~{\rm exp}\left(-c'\sqrt{\frac{\log x}{h^7}}\right) \right),
			\end{equation}
			where $c'>0$ is an effectively computable constant.
			\item[(b)]
			Under GRH, for any $x \geq 2$, we have
			\begin{equation}\label{evl}
				\pi_{f,g}(x,h)=\delta(h){\pi(x)}+O\left(h^6x^{{1}/{2}} \log(h Nx)\right).
			\end{equation}
		\end{enumerate}
		Here the $O$-constants are absolute in both {(a)} and {(b)}. 
	\end{proposition}
	\begin{proof}
The proof follows from the observation \eqref{pi_definition1}. We apply the respective versions of the Chebotarev Density Theorem, namely \eqref{cdt} and \eqref{CDT under GRH}, for the field $L_h$, the groups $\mathscr A_h$, and the conjugation-stable subgroup $\mathscr C_h$.
Note that the number of conjugacy classes in $\mathscr C_h$ is less than $|\mathscr C_h|\ll h^6$
and $n_{L_h} = |\mathscr{A}_h|\ll h^7$, see \cite[Eq. 3.17]{kk}.
Additionally, for an estimate of $d_{L_h}$, we use Hensel's inequality  \cite[Prop. 5, P. 129]{ser}  \begin{equation}\label{Hensel}
			\log d_{L_h}\le |\mathscr{A}_h| \log(hN |\mathscr{A}_h|). 
	\end{equation}
\end{proof}	

For an upper bound of $\pi_{f,g}(x,h)$, we use the result \eqref{thorner} together with the facts $n_{L_h}\ll h^7$ and  Hensel's inequality \eqref{Hensel}. More precisely, we have
\begin{theorem}\label{cosequence of cdt}
    Suppose $f\in S_{k}(N_1)$ and $g\in S_{k}(N_2)$ are two twist-inequivalent non-CM normalized cuspidal newforms. Let $N={lcm}(N_1,N_2)$ and  $h\ge1$ be an integer. Then there exists an absolute constant $c>0$ such that for  $\log x > ch^7 \log(hN)$
    \[
            \pi_{f,g}(x,h)\ll \delta(h)\pi(x).
    \]
\end{theorem}
        
We also require the following result concerning the estimation of primes at which the sum of the Fourier coefficients of the normalized newforms $f$ and $g$ is zero. The proof of this fact is inspired by Serre
     \cite[p. 175]{ser}.
     \begin{lemma}\label{same_coefficients} Let $f$ and $g$ be two non-CM newforms as above, then
	\begin{equation}
		 |\{p \le x: a_f(p)+a_g(p)=0  \}| =\begin{cases}
			O\left( \frac{x}{(\log x)^{1+\ve}}\right), & {\rm ~for ~some~} \ve >0,\\
			O(x^{13/14}), & {\rm ~under ~GRH}.
		\end{cases}
	\end{equation}
    \end{lemma}
\begin{proof}
The proof of the above assertion under the GRH can be found in \cite[Prop. 4.2]{kk}. In this work, however, we present a proof for the unconditional case, where the assumption of GRH is not required. To begin, we first observe that for any prime $\ell$
$$|\{p \le x: a_f(p)+a_g(p)=0  \}|\le \pi_{f,g}(x,\ell)+O(1).$$ Therefore, from Proposition \ref{cdt v1} and using the fact that the Landau-Siegel zero $\beta<1$, we have 
    \[
        |\{p \le x: a_f(p)+a_g(p)=0  \}|\le \frac{x}{\ell\log x}+O\left( \ell^6x ~{\rm exp}\left(-c'\sqrt{\frac{\log x}{\ell^7}}\right) \right).
    \]
 Now we chose a prime $\ell$ between $(\log x)^{1/7}$ and $2(\log x)^{1/7}$, such a prime exists by Bertrand’s postulate, to conclude the claimed assertion.
 \end{proof}

 This section ends with a result about $\pi_{f,g}^*(x,h)$, where for $h \in \N$ and
	\begin{equation}\label{pi_f *}
		\pi_{f,g}^*(x,h):=|\{p\le x:  a_f(p)+ a_g(p)\neq 0 ,~
		a_f(p)+a_g(p)\equiv 0 \pmod h\}|.
	\end{equation}
Combining Proposition \ref{asymptotic_pi} and Lemma \ref{same_coefficients}, we deduce the following.
	\begin{proposition}\label{asymptotic_pi*}
		Let $f$ and $g$ be as in Proposition \ref{asymptotic_pi}.
		Then
		\begin{enumerate}
			\item[(a)]
			for  $\log x \gg h^{21}(\log (hN))^2$ and some $\epsilon>0$, we have
			\begin{equation*}
				\pi_{f,g}^*(x,h)=\delta(h) {\rm Li}(x) +O\left(\delta(h) {\rm Li}(x^\beta)\right)+O\left( h^6x ~{\rm exp}\left(-c'\sqrt{\frac{\log x}{h^7}}\right) \right) +O\left( \frac{x}{(\log x)^{1+\ve}}\right). 
			\end{equation*}
			\item[(b)]
			Under GRH, 
				$\pi_{f,g}^*(x,h)=\delta(h)\pi(x)+O\left(h^6x^{{1}/{2}}\log (hNx)\right)+O(x^{{13}/{14}}).$
		\end{enumerate}
	\end{proposition}

\section{Auxiliary Results for the Proof of Theorem \ref{main}}
In this section, we establish several key results essential for the proof. First, using the Joint Sato-Tate theorem \eqref{joint sato tate} and the dominated convergence theorem, we derive a lower bound on the frequency of primes where the sum of the Fourier coefficients of two twist-inequivalent non-CM normalized newforms is large.
 \begin{proposition}\label{existence of M}
  Let $f\in S_k(N_1)$ and  $g\in S_k(N_2)$ be two non-CM normalized newforms that are twist-inequivalent and have 
  Fourier coefficients $a_f(n)$ and $a_g(n)$, respectively.
     Then for every $\eta\in (0,1]$,  there exists a constant $M>\frac{1}{4}$ such that 
     \[
           \left\vert \left\{ 
                   p\le x  : |a_f(p)+a_g(p)| \ge \frac{p^{(k-1)/2}}{M}
            \right\}\right\vert > \left(1-\frac{\eta}{4}\right)\pi(x)
     \] for all sufficiently large $x$. 
 \end{proposition}

 \begin{proof}
Suppose, for the sake of contradiction, that the conclusion of the statement is not true. This means that for any $M>\frac{1}{4}$ and $x\gg 1$, we have
     \begin{equation}\label{contradiction}
        \left\vert \left\{p\le x : |a_f(p)+a_g(p)|\ge \frac{p^{(k-1)/2}}{M}\right\}\right\vert\le \left(1-\frac{\eta}{4}\right)\pi(x).
     \end{equation}
Consider the rectangle $R = [-2,2] \times [-2,2]$ in $\R^2$. For $r > 0$, let
      $$  
B_r = \left\{(x,y)\in \R^2 : x+y<-\frac{1}{r}\right\}\cap R {\rm \quad and\quad } 
B_r' = \left\{(x,y)\in \R^2 : x + y> \frac{1}{r}\right\}\cap R.$$
Define $$T_M(x)=\left\{p\le x : |a_f(p)+a_g(p)|\ge \frac{p^{(k-1)/2}}{M}\right\}.$$ 
Observe that $T_M(x)$ is a disjoint union of
$$\left\{p\le x : \left(\frac{a_f(p)}{p^{(k-1)/2}}, \frac{a_g(p)}{p^{(k-1)/2}}\right)\in B_M\right\} \quad {\rm{and}}\quad
\left\{p\le x : \left(\frac{a_f(p)}{p^{(k-1)/2}}, \frac{a_g(p)}{p^{(k-1)/2}}\right)\in B_{M}'\right\}.$$
This observation, together with the joint Sato-Tate \eqref{joint sato tate}, gives
 \begin{equation*}\label{sum of measures}
      \lim_{x\rightarrow \infty}\frac{\vert T_M(x)\vert}{\pi(x)} = \nu_{ST}(B_M) + \nu_{ST}(B_M').
 \end{equation*}
This equality holds for every $M$, in particular, for each $n\in \N$. Thus
\begin{equation}\label{sum of measure1}
      \lim_{x\rightarrow \infty}\frac{\vert T_n(x)\vert}{\pi(x)} = \nu_{ST}(B_n) + \nu_{ST}(B_n').
 \end{equation}
The probability measure $\nu_{ST}$ is given as an integral, using the dominated convergence theorem, we have
 $$\lim_{n\rightarrow \infty}(\nu_{ST}(B_n) + \nu_{ST}(B_n'))=\nu_{ST}(R)=1.$$
 Hence, \eqref{sum of measure1} yields
 $$\lim_{n \rightarrow \infty} \lim_{x\rightarrow \infty}\frac{\vert T_n(x)\vert}{\pi(x)}=1.$$ 
 On the other hand, from \eqref{contradiction}, we obtain
 $$\lim_{n \rightarrow \infty} \lim_{x\rightarrow \infty}\frac{\vert T_n(x)\vert}{\pi(x)}\leq  \left(1-\frac{\eta}{4}\right).$$ 
 These two statements can not hold simultaneously, since $\eta>0$ and this leads to a contradiction.
\end{proof}

Using Proposition \ref{existence of M}, we establish an important result that asserts that for two non-CM normalized newforms, the sum of their $p$th Fourier coefficients exhibits unbounded logarithmic growth along a subset of primes $p$ having a positive upper density. 
 
\begin{proposition}\label{unbounded}
 Let $f$ and $g$ be as in Proposition \ref{existence of M}. Then
for any set $S\subset \{p : a_f(p)+ a_g(p)\neq 0 \}$ of primes having a  positive upper density, there exists a strictly increasing sequence of positive integers $\{x_n\}$ such that
     \[
            \sum\limits_{p\in S(x_n)} \log |a_f(p)+a_g(p)|\gg x_n.
     \]
 \end{proposition}

 \begin{proof}
The set $S$ has a positive upper density, therefore, assume
     $
             \limsup \limits_{x\rightarrow \infty} \frac{\vert S(x)\vert}{\pi(x)} = \eta>0.
     $
     This ensures the existence of an increasing sequence ${x_n}$ of positive integers such that
    \begin{equation*}\label{limsup increasing sequence}
        \frac{\vert S(x_n)\vert}{\pi(x_n)} \ge\frac{\eta}{2}, \quad {\rm for\,\, all}~ n.
    \end{equation*}
For such $\eta$, by Proposition \ref{existence of M}, there exists $M> \frac{1}{4}$ such that for sufficiently large $x$, 
\begin{equation}\label{cardinality of T}
        \left\vert \left\{ 
                 p\le x  : |a_f(p)+a_g(p)| \ge \frac{p^{(k-1)/2}}{M}
            \right\}\right\vert > \left(1-\frac{\eta}{4}\right)\pi(x).
      \end{equation}
Let $T\coloneqq \left\{p  : |a_f(p)+a_g(p)| \ge \frac{p^{(k-1)/2}}{M}
            \right\},$ then
            from \eqref{cardinality of T}, we have 
            $$\lvert T(x)\rvert > \left(1-\frac{\eta}{4}\right)\pi(x).$$
Set $
         W = \left\{ p: p > M^2,\, p\in S\cap T  \right\}.$
  Applying \eqref{cardinality of T} and using $\lvert S(x_n)\rvert > \frac{\eta}{2}\pi(x_n)$, we see that the set $W$ has a positive density. More precisely, for all sufficiently large $n$, we have
\begin{equation}\label{cardinality of W}
   \vert W(x_n)\vert >\frac{\eta}{4}\pi(x_n).
\end{equation}
Thus 
\begin{align*}
                \sum\limits_{p\in S(x_n)} \log |a_f(p)+a_g(p)| &\ge \sum\limits_{p\in W(x_n)} \log |a_f(p)+a_g(p)|\\
                &\ge \frac{k-1}{2} \sum_{{p\in W(x_n)}}\log p - \lvert W(x_n)\rvert\log M \\
                &\ge \frac{k-1}{2} \sum_{\substack{p\in W(x_n)}}\log p + O(\pi(x_n)\log M). 
                 \end{align*}
Abel’s summation formula and positivity of density \eqref{cardinality of W},  give
  \begin{align*}
      \sum\limits_{p\in S(x_n)} \log |a_f(p)+a_g(p)| 
      &\gg x_n,
  \end{align*}
  which completes the proof.
  \end{proof}

 \section{Proof of Theorem \ref{main}}
   For $\epsilon>0$, consider the set
$$
        A := \left\{ p : a_f(p)+a_g(p)\neq 0, P(a_f(p)+a_g(p))\le (\log p)^{1/14} (\log \log p)^{3/7-\epsilon}\right\}.
    $$
To complete the proof, we must show that $A$ has zero density.
We proceed by contradiction, that is, assume the set $A$ has a positive upper density. 
    Define 
    $${Q}_{A}(x):=\left\{\ell~ {\rm{prime}}:\ell\mid \prod _{p\in A(x)}{(a_f(p)+a_g(p))}\right\}.$$ 
    If $\nu_\ell(n)$ denotes the $\ell$-adic valuation of an integer $n$,
then for a fixed prime $\ell,$
    \begin{align*}
    \sum\limits_{p\in A(x)} {\nu_\ell}(a_f(p)+a_g(p))\le \sum\limits_{\substack{p\le x\\a_f(p)+a_g(p)\neq 0}} {\nu_\ell}(a_f(p)+a_g(p))=\sum\limits_{\substack{p\le x\\a_f(p)+a_g(p)\neq 0}} \sum\limits_{\substack{m\ge 1\\ {\ell^m}\mid (a_f(p)+ a_g(p)) }} 1.
    \end{align*}
    Applying Deligne's bound \eqref{deligne bound} and interchanging the summations, we obtain
\begin{align*}
             \sum\limits_{p\in A(x)} {\nu_\ell}(a_f(p)+a_g(p))\le \sum\limits_{1\le m\le \frac{\log(4x^{(k-1)/2})}{\log \ell}} \sum\limits_{\substack{p\le x\\a_f(p)+a_g(p)\neq 0\\\ell^m\mid a_f(p)+ a_g(p) }}1=\sum\limits_{1\le m\le \frac{\log(4x^{(k-1)/2})}{\log \ell}} \pi_{f,g}^{*}(x,\ell^m).
    \end{align*}
 By Theorem \ref{cosequence of cdt}, there exists a constant $c>0$ such that
 for $1\le d\le c\frac{(\log x)^{1/7}}{(\log \log x)^{1/7}}$, we have  
    \begin{equation}\label{thorner1}
            \pi_{f,g}^{*}(x,d)\ll \delta(d) \pi(x). 
    \end{equation}
Set $z = c\frac{(\log x)^{1/7}}{(\log \log x)^{1/7}}$  and $m_0 = \left\lfloor{\frac{\log z}{\log \ell}}\right\rfloor$.
Decomposing the sum  \begin{equation}\label{dividing the sum}
        \sum\limits_{1\le m\le \frac{\log(4x^{(k-1)/2})}{\log \ell}} \pi_{f,g}^{*}(x,\ell^m) = \sum\limits_{1\le m\le m_0} \pi_{f,g}^{*}(x,\ell^m) + \sum\limits_{m_0< m\le \frac{\log(4x^{(k-1)/2})}{\log \ell}} \pi_{f,g}^{*}(x,\ell^m).
    \end{equation}
For the first sum, $\ell^m\le c \frac{(\log x)^{1/7}}{(\log \log x)^{1/7}}$, hence from \eqref{thorner1} and  Theorem \ref{cosequence of cdt}, we obtain
    \[
        \sum\limits_{1\le m\le m_0} \pi_{f,g}^{*}(x,\ell^m) \ll \sum\limits_{1\le m\le m_0} \delta(\ell^m)\pi(x) \ll \sum\limits_{1\le m\le m_0} \frac{\pi(x)}{\ell^m}\ll \frac{\pi(x)}{\ell}.
    \]
For the second sum, since $\pi_{f,g}^{*}(x,\ell^m)\le \pi_{f,g}^{*}(x,\ell^{m_0})$ for $m\ge m_0$,
    \begin{align*}
        \sum\limits_{m_0< m\le \frac{\log(4x^{(k-1)/2})}{\log \ell}} \pi_{f,g}^{*}(x,\ell^m) 
        \le \pi_{f,g}^{*}(x,\ell^{m_0})\sum\limits_{m_0< m
        \le \frac{\log(4x^{(k-1)/2})}{\log \ell}} 1.
          \end{align*}
Since $m_0\le \frac{\log z}{\log \ell}$ implies $\ell^{m_0}\le c\frac{(\log x)^{1/7}}{(\log \log x)^{1/7}}$. Therefore using \eqref{thorner1} and $\sum\limits_{m_0< m
        \le \frac{\log(4x^{(k-1)/2})}{\log \ell}} 1 \le \frac{\log x}{\log {\ell}},$
        we obtain
$$\sum\limits_{m_0< m\le \frac{\log(4x^{(k-1)/2})}{\log \ell}} \pi_{f,g}^{*}(x,\ell^m) \ll \frac{\delta(\ell^{m_0})\pi(x)\log x}{\log \ell} .$$
Further, applying Proposition \ref{asymptotic_delta}  and the fact $\frac{\log z}{\log \ell}-1\leq m_0$, the above inequality is
$$\ll \frac{\pi(x)}{\ell^{m_0}}\frac{\log x}{\log \ell}\\ \ll\frac{x}{z} \frac{\ell}{\log \ell}. 
  $$
  Combining the estimates for both sums in \eqref{dividing the sum},
\begin{equation*}\label{bound of nu_l}
     \sum\limits_{p\in A(x)} {\nu_\ell}(a_f(p)+a_g(p))\ll  \frac{x}{z} \frac{\ell}{\log \ell}. 
\end{equation*}
This analysis together with the definition of  $Q_A(x)$ yield 
\[
    \sum\limits_{p\in A(x)}\log |a_f(p)+a_g(p)| =  \sum\limits_{\ell\in Q_A(x)} \log \ell \left(\sum\limits_{p\in A(x)} {\nu_\ell}(a_f(p)+a_g(p))\right) \ll \frac{x}{z}\sum\limits_{\ell\in Q_A(x)}  \ell.
\]
 Set $y = (\log x)^{1/14} (\log \log x)^{3/7-\epsilon}$. Then by the definition of $A$,  $\ell\le y$ whenever $\ell\in Q_A(x)$. Hence
\[
\sum\limits_{p\in A(x)}\log |a_f(p)+a_g(p)|\ll 
 \frac{x}{z}\sum\limits_{\ell\le y} \ell \ll \frac{x}{z} \frac{y^2}{\log y} \ll \frac{x}{(\log \log x)^{2\epsilon}},
\]
for all large $x$. As the set $A$ has  a positive upper density, the last inequality contradicts Proposition \ref{unbounded} and therefore we complete the proof.

\section{Proof of Theorem \ref{over n}}
For a given $\epsilon>0$, consider the set
  $$\mathcal{P}=\{p:a_f(p)+ a_g(p)\neq 0,\,P(a_f(p)+a_g(p))>(\log p)^{1/14}(\log \log p )^{3/7-\epsilon} \}.$$
  From Theorem \ref{main}, the set $\mathcal{P}$ has density 1. For the proof, we use this set of primes to construct a set of natural numbers of density 1 and have the property as stated.
  Define
  $$\mathcal{Q}=\{n: \exists ~ p\in \mathcal{P} ~{\rm{such ~that}}~p|n,\, p>n^{1/((\log\log n)^{\epsilon})}\}. 
  $$
  We claim that $\mathcal{Q}$ has natural density equal to $1$. Let $x$ be sufficiently large and $x_0=x^{1/(\log\log x)^{\epsilon}}$. We apply Brun's sieve \cite[Thm. 2.1, P. 57]{HR} to say that 
  \begin{equation}\label{brun}
  \left\vert\left\{n\leq x: \bigg(n,\prod_{\substack{ p\in \mathcal{P}\\x_0<p<x^{1/5}}}p\bigg)=1\right\}\right\vert\ll x\prod_{\substack{x_0<p<x^{1/5}\\ p\in \mathcal{P}}}\bigg(1-\frac{1}{p}\bigg)=O(\pi(x)).\end{equation}
  Additionally, one can easily verify that the set
  $$\{n: \sqrt{x}\le n \le x, \, n\not\in \mathcal{Q} \}$$
  is contained inside the leftmost set in \eqref{brun}, and hence using the inequality in \eqref{brun}, we obtain 
  \begin{equation}\label{brun2}
      \vert\{n: \sqrt{x}\le n \le x, \, n\not\in \mathcal{Q} \}\vert = o(x).
  \end{equation}
  This implies
  $$
        \sum_{\substack{n\le x\\ n\not \in \mathcal{Q} }}1 = \sum_{\substack{n\le \sqrt{x}\\ n\not\in \mathcal{Q} }}1 + \sum_{\substack{\sqrt{x}\le n\le x\\ n\not\in \mathcal{Q} }}1 \le \sqrt{x} + o(x) = o(x).
$$
Thus the set $\mathcal{Q}$ has  natural density $1$. Now consider the following subset of $\mathcal{Q}$
\[
    \widetilde{\mathcal{Q}}:= \left\{n : p^2\mid n ~{\rm for ~some~} p\in \mathcal{P} ~{\rm and~} p> n^{\frac{1}{(\log \log n)^\epsilon}}\right\}.
\]
Then 
\[
    |\widetilde{\mathcal{Q}}(x)| = \sum\limits_{\substack{n\le \sqrt{x}\\ p\in \mathcal{P}, p^2\mid n\\ p>n^{\frac{1}{(\log \log n)^\epsilon}}}} 1 + \sum_{\substack{\sqrt x\le n \le x\\  p\in \mathcal{P}, p^2\mid n\\ p>n^{\frac{1}{(\log \log n)^\epsilon}} }}1 < \sqrt{x} + S_1,
\]
where $S_1$ denotes the second sum. Changing the order of summation in $S_1$ and realizing {that $\sqrt{x_0}< n^{\frac{1}{(\log \log n)^\epsilon}}$ for $n\in [\sqrt{x},x]$}, we obtain
\[
        S_1 \le \sum_{p\ge \sqrt{x_0}} \sum_{\substack{\sqrt x\le n \le x\\  p^2\mid n}} 1 \le \sum_{\sqrt{x_0} \le p\le \sqrt{x}} \frac{x}{p^2}= o(x).
 \]
Hence, if we define
$$T=\left\{n: p||n~~{\rm  {for~ some~}} p \in \mathcal P ~{\rm and}~ p> n^{\frac{1}{(\log \log n)^\epsilon}}\right \}$$
where $p||n$ means $\left(\frac{n}{p},p\right)=1$,  then the above analysis give that the subset $T$ of $\mathcal Q$
has natural density equal to 1. 

Now, for $n\in T$, if $(a_f\ast g_f)(n)\neq 0$, then using the fact that convolution of two multiplicative functions is again  multiplicative, we have 
    \[
    P((a_f\ast a_g)(n)) \ge P(a_f(p)+a_g(p)),
    \]
    where  $p\in \mathcal{P}$  with $p\mid\mid n$ and $p> n^{\frac{1}{(\log \log n)^\epsilon}}.$ Thus
    $
         P((a_f\ast a_g)(n)) \ge (\log n)^{1/14} (\log \log n)^{3/7-\epsilon},
    $
    for all sufficiently large $n \in T$ and this completes the proof.

\section{Proof of Corollary \ref{character}}
 Suppose the conclusion in Corollary \ref{character} is not true, i.e., $f$ and $g$ are twist-inequivalent. From  Theorem \ref{main}, it follows that for the given $\epsilon>0$, the set of primes \[
 \{p: |a_f(p)+a_g(p)| \leq (\log p)^{1/14} (\log \log p)^{3/7-\epsilon}\}
\]
has density 0 which is a contradiction to the fact that the set has a positive upper density and hence $f= \chi \otimes g$ for some Dirichlet character $\chi$. Now $f$ and $g$ have trivial nebentypus implies $\chi^2=1$.

\section{Preparation for the proof of Theorem \ref{under GRH}} \label{section: normal order}
In this section, we study the normal order of the prime factors of $a_f(p)+a_g(p)$
for two twist-inequivalent normalized newforms of weight $k\ge 2$. These results play a key role in the proof of Theorem \ref{under GRH} and are also of independent interest.
In \cite{mmp}, it was shown that for normalized newforms of weight $k=2$, the normal order of the prime factors of $a_f(p)+a_g(p)$ is $\log \log p$. Our main result extends this conclusion to normalized newforms of arbitrary weight.

To state this result precisely, let $\omega(n)$ denote the number of distinct prime divisors of $n$ and for $u>0$, let $\omega_u(n)$ count the distinct prime divisors of $n$ up to $u$. We establish the following result.

\begin{theorem}\label{normal order of sums}
Let $f\in S_k(N_1)$ and  $g\in S_k(N_2)$ be two non-CM normalized newforms of weight $k\ge 2$ with trivial nebentypus and integer Fourier coefficients $a_f(n)$ and $a_g(n)$, respectively. 
If $f$ and $g$ are twist-inequivalent, then under GRH, for sufficiently large $x$, we have
\begin{enumerate}
    \item 
    for any $u=x^{\eta}$ with $0<\eta< \frac{1}{14}$, 
     \[
        \sum\limits_{\substack{p\le x\\a_f(p)+a_g(p)\neq 0}} (\omega_u(a_f(p)+a_g(p))-\log \log u)^2 = O(\pi(x) \log \log u).  
    \]
    \item 
     $
        \sum\limits_{\substack{p\le x\\a_f(p)+a_g(p)\neq 0}} (\omega(a_f(p)+a_g(p))-\log \log x)^2 = O(\pi(x) \log \log x). $
\end{enumerate}
\end{theorem}

\begin{proof}
The proof follows the method of \cite{mmp}, which is inspired by Turán's original ideas \cite{Turan}. We show that the sequence $a_f(p)+a_g(p)$ and its counting function $\pi_{f,g}$, defined in \eqref{pi_definition}, satisfy all conditions from \cite[Sec. 2]{mmp}.
\begin{enumerate}
    \item[(1)] 
    Using \eqref{deligne bound}, for a prime $p$,
$|a_f(p)+a_g(p)|\le 4p^{\frac{k-1}{2}}$.
\item[(2)] 
By Lemma \ref{same_coefficients}, under GRH,
    $\lvert \{p\le x: a_f(p) +a_g(p)=0\}\rvert = O(x^{1-\eta})$ with $\eta < 1/14$.
\item[(3)] 
From Proposition \ref{cdt v1}, under GRH, we see that 
\[  
    \sum\limits_{\ell\le x^{\eta}} |\pi_{f,g}(x,\ell) - \delta(\ell)\pi(x) |\ll \frac{x }{\log x}. 
    \]
\item[(4)] By the first part of Proposition \ref{asymptotic_delta},
$\sum_{\ell}\left\lvert\delta(\ell)-\frac{1}{\ell}\right\rvert=O(1).$
\item[(5)]
Using Proposition \ref{delta_multiplicative_large_prime} and Proposition \ref{cdt v1}, under GRH,
\[  
    \sum\limits_{\ell_1, \ell_2\le x^{\eta_1}, \ell_1\neq \ell_2} |\pi_{f,g}(x,\ell_1\ell_2) - \delta(\ell_1)\delta(\ell_2)\pi(x) |\ll \frac{x }{\log x}, \quad \text{ for } \eta_1 < 1/28. 
    \]
\end{enumerate}
Since all hypotheses of \cite[Sec. 2]{mmp} are satisfied, the proof is complete.
\end{proof}

As a consequence of Theorem \ref{normal order of sums}, we prove that $\omega(a_f(p)+a_g(p))$ has normal order $\log \log p$, where $p$ runs over the set of primes. More precisely,
\begin{corollary}\label{over p}
With notation as above, we have
 \begin{equation}\label{proposition 5.5 rammurty}
     \sum\limits_{\substack{p\le x\\a_f(p)+ a_g(p)\neq 0}} (\omega(a_f(p)+a_g(p))-\log \log p)^2  = O(\pi(x) \log \log x). 
 \end{equation}
 \end{corollary}
\begin{proof}
For the proof, we begin by expressing
     $$\omega(a_f(p)+a_g(p))-\log \log p = (\omega(a_f(p)+a_g(p))-\log \log x) + (\log \log x - \log \log p).$$ 
     This representation allows us to rewrite the left-hand side of the claimed sum as
     \begin{align*}\label{eq 2:proposition 5.5 rmv}
       \sum\limits_{\substack{p\le x\\a_f(p)+a_g(p)\neq 0}} (\omega(a_f(p)+a_g(p))-\log \log x)^2 + \sum\limits_{\substack{p\le x\\a_f(p)+a_g(p)\neq 0}} (\log \log x - \log \log p)^2 \\  \nonumber + 2 \sum\limits_{\substack{p\le x\\a_f(p)+a_g(p)\neq 0}} (\omega(a_f(p)+a_g(p))-\log \log x) (\log \log x - \log \log p).
        \end{align*}
Next, using  Theorem \ref{normal order of sums}, the first sum is bounded by 
 $\pi(x) \log \log x $, while an elementary argument shows that the second sum satisfies $\ll \pi(x)$. 
 Applying the Cauchy-Schwarz inequality to the third sum, we find that it satisfies the desired estimate, completing the proof.
\end{proof}

\section{Proof of Theorem \ref{under GRH}}\label{sec: proof uner grh}
Let $F$ 
be a monotone increasing function on $\R$, chosen later, with $F(x)> 14$ and $F(x) = O(\log x)$. Define $$y=y(x)=x^{\frac{1}{F(x)}}.$$
 From Theorem \ref{normal order of sums}, we have 
    \begin{equation}\label{normal order}
    \sum\limits_{\substack{x/2\le p\le x\\a_f(p)+ a_g(p)\neq 0}} (\omega_y(a_f(p)+a_g(p))-\log \log y)^2 \ll \pi(x)\log \log y.
    \end{equation}
This implies that for every $\delta >0$, the number of primes $\frac{x}{2}\leq p \leq x$ that do not satisfy the inequality 
    \begin{equation}\label{normal order 2}
      |\omega_y(a_f(p)+a_g(p))-\log \log y| < (\log \log y)^{1/2+\delta}
    \end{equation}
is $o(\pi(x))$.
   For a prime $p \in [\frac{x}{2},x]$ which satisfies the inequality \eqref{normal order 2},  let  $z= p^{1/F(p)}$. Hence, for the given  $\epsilon$ and for almost all $p\in [\frac{x}{2},x]$, 
 $$\omega_y(a_f(p)+a_g(p))\leq \log \log y +(\log \log y)^{1/2+\epsilon}.$$
Since  $F$ is an increasing function with $F(x)>14$, it follows that $\log \log z \leq \log \log p$. Combining this with the above inequality and together with the fact that at the prime $p,~ y=z$, we obtain
 $$\omega_z(a_f(p)+a_g(p))\leq \log \log z +(\log \log p)^{1/2+\epsilon}.$$
Using Corollary \ref{over p} and a argument similar to \eqref{normal order 2}, we also have
$$\omega(a_f(p)+a_g(p))> \log \log p -(\log \log p)^{1/2+\epsilon}$$
for almost all $p$. 
Thus, for almost all $p$, the number of distinct prime divisors of $a_f(p)+a_g(p)$ that are greater than $z$ is at least
$$\log \log p -(\log \log p)^{1/2+\epsilon}-\log \log z -(\log \log p)^{1/2+\epsilon}=\log \frac{\log p}{\log z}-2(\log \log p)^{1/2+\epsilon}.$$

Now, choose $F$ such that $\log F(p)=3(\log \log p)^{1/2+\epsilon}$. Using the definition of $z$, we obtain
\begin{equation}\label{final}
    \log \frac{\log p}{\log z}-2(\log \log p)^{1/2+\epsilon}=\log F(p)-2(\log \log p)^{1/2+\epsilon}=(\log \log p)^{1/2+\epsilon}.
\end{equation}
Since this quantity is positive, it follows that 
 $P(a_f(p)+a_g(p))\geq z$. This completes the proof of part $(a)$.

For part $(b)$, the conclusion follows from the fact that $|a_f(p)+a_g(p)|$ has at least $(\log \log p)^{1/2+\epsilon}$ many distinct prime divisors greater than $
z$, as established in \eqref{final}.

\section*{Acknowledgement}
The first author’s research is supported by the Science and Engineering Research Board, India, through grant SRG/2023/000228. 

\bibliography{ref}
\bibliographystyle{alpha}

\end{document}